\documentclass[11pt]{amsart}
\usepackage{amsmath,amstext,amsbsy,amssymb,amscd}
\usepackage{amsmath}
\usepackage{amsxtra}
\usepackage{amscd}
\usepackage{amsthm}
\usepackage{amsfonts}
\usepackage{graphicx}%
\usepackage{amssymb}
\usepackage{eucal}
\usepackage{color}
\usepackage{multirow}
\usepackage[enableskew,vcentermath]{youngtab}

\newtheorem{thm}{Theorem}[section]
\theoremstyle{plain}

\newtheorem{lem}[thm]{Lemma}
\newtheorem{prop}[thm]{Proposition}

\theoremstyle{definition}

\theoremstyle{remark}
\newtheorem{rem}[thm]{Remark}

\newtheorem*{thma}{{\bf Theorem A}}
\newtheorem*{thmb}{{\bf Theorem B}}
\newtheorem*{thmc}{{\bf Theorem C}}
\definecolor{A}{rgb}{.75,1,.75}

\numberwithin{equation}{section}

\newcommand{\ds}{\displaystyle}

\newcommand{\Z}{\mathbb Z}

\newcommand{\F}{\mathbb F}

\newcommand{\St}{{\rm St}}
\newcommand{\sym}{S^{\bullet}(V)}
\newcommand{\wedg}{\wedge^{\bullet}(V)}
\newcommand{\Det}{{\rm Det}}

{\vskip-\lastskip\medskip
  \noindent
  {\em #1.}\enspace
  }%
{\qed\par\medskip
  }

\begin{document}

\title[Dickson-Mui invariants and Steinberg module multiplicity]
{Twisted Dickson-Mui invariants and the Steinberg module
multiplicity}
\author[Wan and Wang]{Jinkui Wan and Weiqiang Wang}

\address{
(Wan) Department of Mathematics,
Beijing Institute of Technology,
Beijing, 100081, P.R. China. }
\email{wjk302@gmail.com}

\address{
(Wang) Department of Mathematics, University of Virginia,
Charlottesville,VA 22904, USA.}
\email{ww9c@virginia.edu}


\begin{abstract}
We determine the invariants, with arbitrary determinant twists, of
the parabolic subgroups of the finite general linear group
$GL_n(q)$ acting on the tensor product of the symmetric algebra
$\sym$ and the exterior algebra $\wedg$ of the natural
$GL_n(q)$-module $V$. In addition, we obtain the graded
multiplicity of the Steinberg module of $GL_n(q)$ in
$\sym\otimes\wedg$, twisted by an arbitrary determinant power.
\end{abstract}

\maketitle



\section{Introduction}

\subsection{}
Let $p$ be a fixed prime and $\F_q$ be the finite field with
$q=p^r$ for some $r\geq 1$. The finite general linear group
$GL_n(q)$ acts naturally on the symmetric algebra $\sym$, where
$V=\F_q^n$ is the standard $GL_n(q)$-module. It is a classical
theorem of Dickson~\cite{D} that the algebra of
$GL_n(q)$-invariants in $\sym$ is a polynomial algebra in $n$
generators (called Dickson invariants), and  its Hilbert
series is given by
\begin{equation*}
H \big(\sym^{GL_n(q)};t \big)
=\frac{1}{\prod^{n-1}_{j=0}(1-t^{q^{n}-q^{j}})}.
\end{equation*}
Let $P_I$ be the parabolic subgroup associated to a
composition $I=(n_1,n_2,\ldots,n_{\ell})$ of $n$, cf.~
\eqref{parabolic}. Generalizing \cite{D},
Kuhn and Mitchell~\cite{KM} showed
that the algebra $\sym^{P_I}$ of
$P_I$-invariants in $\sym$ is a polynomial algebra in $n$
explicit generators.

In another classical work \cite{Mui}, Mui determined the
$GL_n(q)$-invariants in the tensor product $\sym\otimes\wedg$,
where $\wedg$ denotes the exterior algebra of $V$. Actually Mui
only formulated his invariant results for $q=p$ (i.e., $r=1$) as
he was mainly interested in the connections with the cohomology
groups of $(\Z/p\Z)^n$ and of symmetric groups, but his algebraic
proof remains valid for a general $q$. More recently, Minh and
T\`{u}ng \cite{MT} determined the  $P_I$-invariants in
$\sym\otimes\wedg$, again in the case $q=p$, as they needed to use
some Steenrod algebra arguments. In any event, topologists (with
the exception of \cite{KM}) are mostly content with establishing
various invariant results for the prime field $\F_p$, which are
all they needed for the topological applications. We refer to
Wilkerson \cite{Wi} and Smith \cite{Sm} for expositions on
Dickson invariants and  connections to topology.

\subsection{}

The first goal of this paper is to study the generalizations of
Dickson-Mui invariants over a  finite field $\F_q$ in a general
framework with arbitrary twists by the $GL_n(q)$-determinant
module $\rm Det$. Our first main result is the following.
\begin{thma}   \label{th:free}
Let $I=(n_1,n_2,\ldots,n_{\ell})$ be a composition of $n$. Then,
for $0 \leq k\leq q-2$, $(\sym\otimes\wedg\otimes {\rm
Det}^k)^{P_I}$ is a free module of rank $2^n$ over the algebra
$(\sym)^{P_I}$.
\end{thma}
We refer to Theorem~\ref{thm:invparabolic} for a more precise
version of Theorem~A where an explicit basis for the free module
is given. Theorem~A in the special case when $k=0$ and $q=p$ is
due to Minh and T\`{u}ng, and our approach toward Theorem~A (or
rather Theorem~\ref{thm:invparabolic}) is a generalization of
\cite{MT}, which is in turn built heavily on \cite{Mui}. In the
case for $k=0$, we supply an elementary proof for the part of
arguments in \cite{MT} which used Steenrod algebra, hence removing
their restriction on $q=p$. Hence, we easily obtain in
Theorem~\ref{th:KI-inv} a generalization to a general $q$ of the
results of \cite{MT} on invariants by other distinguished
subgroups of $GL_n(q)$. The new cases with nontrivial twists by
$\rm Det$ for $1\leq k\leq q-2$ are motivated by Theorem~C below
and used in its proof. (Note that $\Det^{q-1}$ is trivial.)
Theorem~B indicates that the answers in the twisted cases are
actually simpler.

\begin{thmb}
Let $1\leq k\leq q-2$. Then, the Hilbert series (in variables $t$
and $s$) of the bi-graded vector space $(\sym\otimes\wedg\otimes
{\Det}^k)^{P_I}$ is equal to
\begin{align*}
t^{(q-2-k)\frac{q^n-1}{q-1}} \cdot \frac{\prod^{n-1}_{i=0}
(s+t^{q^i})}{\prod^{\ell}_{i=1}\prod^{n_i}_{j=1}(1-t^{q^{m_i}-q^{m_i-j}})}.
\end{align*}
(See \eqref{eq:mi} for the definition of $m_i$ in terms of $I$.)
\end{thmb}
We remark that the case for $k=0$, which is excluded in Theorem~B,
also affords an explicit yet more complicated formula, see
Theorem~\ref{thm:Hser.para.}. Theorem~B and
Theorem~\ref{thm:Hser.para.} will be derived from
Theorem~\ref{thm:invparabolic}.

As observed by Crabb \cite{Cr}, the space $\sym\otimes(\wedg)^*$
is naturally an algebra, and its subalgebra of
$GL_n(q)$-invariants affords a simpler description than in the
original setting of Mui. As an application of Theorem~A, we show
that the subalgebra of $P_I$-invariants in $\sym\otimes(\wedg)^*$
is isomorphic to the tensor product of $\sym^{P_I}$ and an
exterior algebra for any $I$, recovering the main result
in~\cite{Cr} by setting $I=(n)$ and hence $P_I =GL_n(q)$.

\subsection{}

Recall (cf. Humphreys \cite{Hu}) that $GL_n(q)$ affords a
distinguished module $\St$, called Steinberg module, which is
projective and absolutely irreducible. Our second main result
concerns about the bi-graded multiplicity of $\St$ in $\sym\otimes
\wedg \otimes {\rm Det}^k$.
\begin{thmc}
(1) The bi-graded multiplicity of $\St$ in $\sym\otimes \wedg$ is
\begin{align*}
H_{\St}\big(\sym\otimes \wedg;t,s \big) =\ds t^{-n} \cdot\frac{
(st^{q^n-1}+t^{q^{n-1}})
\prod^{n-2}_{i=0}(s+t^{q^i})}{\prod^n_{i=1}(1-t^{q^i-1})}.
\end{align*}

(2) For $1\leq k\leq q-2$, the bi-graded multiplicity of $\St$ in
$\sym\otimes\wedg\otimes {\rm Det}^k$ is
\begin{align*}
H_{\St}\big( \sym\otimes \wedg\otimes {\rm Det}^k;t,s \big) =\ds
t^{-n+(q-1-k)\frac{q^n-1}{q-1}} \cdot
\frac{\prod^{n-1}_{i=0}(s+t^{q^i})}{\prod^n_{i=1}(1-t^{q^i-1})}.
\end{align*}
\end{thmc}

In the special case when $q=p$, Parts~(1) and (2) of Theorem~C
were established by Mitchell-Priddy \cite{MP} and Mitchell
\cite{Mi} respectively, via a topological approach which involves
Steenrod algebra in an essential way. The question of generalizing
to a general $q$ was raised explicitly by Kuhn and Mitchell
\cite{KM}, but it was warned back then that a direct
generalization of their algebraic method could be too complicated,
as it would require one to establish our Theorem~B (which is in
turn built on Theorem~\ref{thm:invparabolic}) first. With
Theorems~B at our disposal, we indeed prove Theorem~C for a
general $q$ following the algebraic approach of \cite{KM}, which
uses a modular version of a theorem of Curtis \cite{Cu} (also see
Solomon \cite{So}).


\subsection{}

From the viewpoint of representation theory, all the works about
various generalizations of Dickson invariants and the Steinberg
module multiplicity are providing very interesting yet limited
answers to the problem of understanding the $GL_n(q)$-module
structure of the symmetric algebra $\sym$. In a separate
publication \cite{WW}, we will readdress such a problem from a completely
different approach which employs the deep connections between
finite groups of Lie type and algebraic groups.
In particular, we will give a new proof of Theorem~C.

\subsection{}
The paper is organized as follows. In Section~\ref{sec:DM inv}, we
review some needed results from \cite{Mui, KM, MT}, and rework
parts of \cite{MT} for a general $q$. Proofs of Theorems~A and B
and the applications to generalization of Crabb's work are given
in Section~\ref{sec:para.inv.}. Theorem~C is then proved in
Section~\ref{sec:Stmult}.

\vspace{.2cm}

{\bf Acknowledgments.} We are grateful to Nick Kuhn, from whom we
first learned about the interests of topologists in Dickson
invariants and Steinberg modules, for stimulating discussions and
generously sharing his expertise. The research of the second
author is partially supported by NSF grant DMS-0800280.


\section{Dickson-Mui invariants and generalizations}\label{sec:DM inv}
In this section, we will recall the work by Dickson and
Kuhn-Mitchell on invariants in $\sym$ and then the work of Mui on
invariants in $\sym\otimes\wedg$. We then modify the approach of
Minh-T\`ung to make it valid for a general $q$. Throughout this
section, the symmetric algebra $\sym$ and the exterior algebra
$\wedg$ will be identified with $\F_q[x_1,\ldots,x_n]$ and
$E[y_1,\ldots,y_n]$, respectively.


\subsection{The invariants of Dickson and Kuhn-Mitchell in $\sym$.}
\label{subsec:Dickinv}

For $1\leq m\leq n$, we define $V_m$ and $L_m$ by letting
\begin{align*}
V_m=V_m(x_1,\ldots,x_m)
 &=\prod_{c_1,\ldots,c_{m-1}\in\F_q}(c_1x_1+\cdots+c_{m-1}x_{m-1}+x_m)
   \\
L_m=L_m(x_1,\ldots,x_m)
 &=V_1V_2\cdots V_m
=\prod^m_{k=1}\prod_{c_1,\ldots,c_{k-1}\in\F_q}
 (c_1x_1+\cdots+c_{k-1}x_{k-1}+x_k).
\end{align*}
By~\cite{D} (cf. \cite[Lemma~2.3]{Mui}), we have
\begin{align*}
L_m=\left|
{\begin{array}{cccc}
 x_1& x_2&\cdots &x_m\\
 x_1^q &  x_2^{q}&\cdots& x_m^{q} \\
 \vdots&\vdots&\vdots&\vdots\\
 x_1^{q^{m-1}}&x_2^{q^{m-1}}&\cdots& x_m^{q^{m-1}}
 \end{array} }
\right|.
\end{align*}
For $0\leq k\leq m$, define $L_{m,k}$ and $Q_{m,k}$ by letting
\begin{align}
L_{m,k}&=
\left|
{\begin{array}{cccc}
 x_1& x_2&\cdots &x_m\\
 x_1^q &  x_2^{q}&\cdots& x_m^{q} \\
 \vdots&\vdots&\vdots&\vdots\\
 \widehat{x_1^{q^k}}&\widehat{x_2^{q^{k}}}&\cdots&\widehat{x_m^{q^{k}}}\\
 \vdots&\vdots&\vdots&\vdots\\
 x_1^{q^m}&x_2^{q^m}&\cdots& x_m^{q^m}
 \end{array} }
\right|,\notag\\
Q_{m,k}&=Q_{m,k}(x_1,\ldots,x_m) = L_{m,k}/L_m,\label{eqn:Dickson}
\end{align}
where the hat $\, \widehat{ } \,$ means the omission of the given
term as usual. It follows from~\cite{D} (cf.
\cite[Proposition~2.6]{Mui}) that $Q_{m,k}$ can be computed
recursively by the following relations:
\begin{align}
Q_{m,0}&=L_m^{q-1},\label{eqn:relQL}\\
Q_{m,k}&=Q_{m-1,k}V_m^{q-1}+Q_{m-1,k-1}^{q}, \notag\\
V_m
&=x_m^{q^{m-1}}+\sum^{m-2}_{k=0}(-1)^{m-1-k}Q_{m-1,k}x_m^{q^k}.
\label{eqn:Vm}
\end{align}

According to Dickson \cite{D}, both subalgebras of invariants over
${SL_n(q)}$ and over ${GL_n(q)}$ in $\F_q[x_1,\ldots,x_n]$ are
polynomial algebras, and moreover,
\begin{eqnarray*}
\begin{split}
\F_q[x_1,\ldots,x_n]^{SL_n(q)}&=\F_q[L_n,Q_{n,1},\ldots,Q_{n,n-1}], \\
\F_q[x_1,\ldots,x_n]^{GL_n(q)}&=\F_q[Q_{n,0},\ldots,Q_{n,n-1}].
\end{split}
\end{eqnarray*}

Let $I=(n_1,n_2,\ldots,n_{\ell})$ be a composition of $n$, i.e.,
$\sum_jn_j=n$, whose parts $n_j$ are  assumed to be all positive
throughout the paper. Denote by $P_I$ the standard parabolic
subgroup of the form
\begin{align}  \label{parabolic}
P_{I}=\left(
{\begin{array}{cccc}
 GL_{n_1}(q) & *&\cdots &* \\
 0 &  GL_{n_2}(q)&\cdots& * \\
 \vdots&\vdots&\vdots&\vdots\\
 0&0&\cdots& GL_{n_{\ell}}(q)
 \end{array} }
\right)
\end{align}
and set
\begin{equation}  \label{eq:mi}
m_0=0, \qquad m_i=\sum^i_{j=1}n_j \quad (1\leq i\leq \ell).
\end{equation}
For $1\leq j\leq n_i$, recalling $Q_{m,k}$ from
(\ref{eqn:Dickson}), we define
\begin{align*}
v_{i,j} &=\prod_{c_1,\ldots,c_{m_{i-1}}\in\F_q}(c_1x_1
+\cdots+c_{m_{i-1}}x_{m_{i-1}}+x_{m_{i-1}+j}),
 \\
q_{i,n_i-j}&=Q_{n_i,n_i-j}(v_{i,1},\ldots,v_{i,n_i}).
\end{align*}

Recall the Hilbert series of a graded space $W^\bullet =\oplus_i
W^i$ is by definition the generating function  $H (W^\bullet, t)
=\sum_i t^i \dim W^i$.

\begin{thm}\label{thm:KM-He}(\cite[Theorem~2.2]{KM}, \cite[Theorem~1.4]{He})
The subalgebra $\sym^{P_I}$ of $P_I$-invariants in $\sym$ is a
polynomial algebra on the generators $q_{i,n_i-j}$ of degree
$q^{m_i}-q^{m_i-j}$ with $1\leq i\leq \ell, 1\leq j\leq n_i$.
Moreover, the Hilbert series of $\sym^{P_I}$ is
\begin{align*}
H\big(\sym^{P_I},t \big)
=\frac{1}{\prod^{\ell}_{i=1}\prod^{n_i}_{j=1}(1-t^{q^{m_i}-q^{m_i-j}})}.
\end{align*}
\end{thm}
%


\subsection{Invariants of Mui in $\sym\otimes\wedg$.}

Let $A=(a_{ij})$ be a $n\times n$-matrix with entries in a
possibly noncommutative ring. The (row) determinant of $A$ is
defined as follows:
\begin{align*}
|A| =\det (A)
=\sum_{\sigma\in S_n} {\rm
sgn}(\sigma)a_{1\sigma(1)}a_{2\sigma(2)}\cdots a_{n\sigma(n)}.
\end{align*}

Suppose $1\leq j\leq m\leq n$ and let $(b_1,\ldots,b_j)$ be a
sequence of integers such that $0\leq b_1<\cdots<b_j\leq m-1$. We
define the element $M_{m;b_1,\ldots,b_j}$ in
$\F_q[x_1,\ldots,x_n]\otimes E[y_1,\ldots,y_n]$ by  the following
determinant of $m\times m$ matrix
\begin{align*}
M_{m;b_1,\ldots,b_j}
:=\frac{1}{j!}~
\left|
\begin{array}{cccc}
y_1&y_2&\cdots&y_m\\
y_1&y_2&\cdots&y_m\\
\vdots&\vdots&\vdots&\vdots\\
y_1&y_2&\cdots&y_m\\
x_1&x_2&\cdots&x_m\\
\vdots&\vdots&\vdots&\vdots\\
\widehat{x_1^{q^{b_1}}}&\widehat{x_2^{q^{b_1}}}&\cdots&\widehat{x_m^{q^{b_1}}}\\
\vdots&\vdots&\vdots&\vdots\\
\widehat{x_1^{q^{b_j}}}&\widehat{x_2^{q^{b_j}}}&\cdots&\widehat{x_m^{q^{b_j}}}\\
\vdots&\vdots&\vdots&\vdots\\
x_1^{q^{m-1}}&x_2^{q^{m-1}}&\cdots&x_m^{q^{m-1}}
\end{array}
\right|,
\end{align*}
where the first $j$ rows $(y_1, y_2, \ldots, y_m)$ are repeated.
Note that, for $ g\in GL_m(q)$,
\begin{align*}
g\cdot M_{m;b_1,\ldots,b_j}=\det(g)\ M_{m;b_1,\ldots,b_j}.
\end{align*}

We now recall the main results in \cite[Chapter I]{Mui} (valid
over $\F_q$ though only formulated over $\F_p$ therein), which
will be used in the modified proof of results of \cite{MT} for a
general finite field $\F_q$ in the next subsection.

For two integer sequences $(b_1,\ldots,b_j)$ and
$(c_1,\ldots,c_j)$, we say $(c_1,\ldots,c_j)>(b_1,\ldots,b_j)$ if
there exists an integer $1\leq k\leq j$ such that $c_k>b_k$ and
$c_i=b_i$ for $k+1\leq i\leq j$. Denote by $U_n(q)$ the subgroup
of $GL_n(q)$ consisting of all uni-uppertriangular matrices.

\begin{lem} (cf. \cite[Proposition~4.5]{Mui})\label{lem:Mui4.5}
Suppose that $1\leq j\leq m\leq n$ and $0\leq b_1<\cdots<b_j\leq m-1$.
Then we have
\begin{align*}
M_{m;b_1,\ldots,b_j}
=(-1)^{\frac{j(j-1)}{2}}M_{m;b_1}M_{m;b_2}\cdots M_{m;b_j}/L_m^{j-1}.
\end{align*}
\end{lem}

\begin{lem}(cf. \cite[Lemma~5.1]{Mui})\label{lem:Mui5.1}
Suppose that $1\leq j\leq m\leq n$ and $0\leq b_1<\cdots<b_j\leq
m-1$. The following holds:
\begin{align*}
M_{m;b_1,\ldots,b_j}V_{m+1}\cdots V_n
=M_{n;b_1,\ldots,b_j}
+\sum_{c_1,\ldots,c_j}M_{n;c_1,\ldots,c_j}\cdot h_{c_1,\ldots,c_j},
\end{align*}
where the summation is over the sequences $(c_1,\ldots,c_j)$
such that $(c_1,\ldots,c_j)>(b_j-j+1,b_j-j+2,\ldots,b_j)$
and $h_{c_1,\ldots,c_j}\in\F_q[x_1,\ldots,x_n]^{U_n(q)}$.
\end{lem}

\begin{lem}\cite[Lemma~5.2]{Mui}\label{lem:Mui5.2}
Fix $1\leq j\leq n$ and let
\begin{align*}
f=\sum^n_{m=j}\sum_{0\leq b_1<\cdots<b_j= m-1}
M_{m;b_1,\ldots,b_j}\cdot g_{b_1,\ldots,b_j}(x_1,\ldots,x_n).
\end{align*}
Then $f=0$ if and only if all $g_{b_1,\ldots,b_j} =0$.
\end{lem}

\begin{thm}(cf. \cite[Theorem~4.8, Theorem~4.17, Theorem~5.6]{Mui})
\label{thm:Muithms}
We have
\begin{align*}
(\sym\otimes\wedg)^{SL_n(q)}
=~&\F_q[L_n,Q_{n,1},\ldots,Q_{n,n-1}]\\
&\oplus\sum^n_{j=1}\sum_{0\leq b_1<\cdots<b_j\leq n-1}
M_{n;b_1,\ldots,b_j}\F_q[L_n,Q_{n,1},\ldots,Q_{n,n-1}], \\
(\sym\otimes\wedg)^{GL_n(q)}
=~&\F_q[Q_{n,0,}Q_{n,1},\ldots,Q_{n,n-1}]\\
&\oplus\sum^n_{j=1}\sum_{0\leq b_1<\cdots<b_j\leq n-1}
M_{n;b_1,\ldots,b_j}L_n^{q-2}\F_q[Q_{n,0},Q_{n,1},\ldots,Q_{n,n-1}],  \\
(\sym\otimes\wedg)^{U_n(q)}
=~&\F_q[V_1,\ldots,V_n]\\
&\oplus\sum^n_{j=1}\sum^n_{m=j}\sum_{0\leq b_1<\cdots<b_j= m-1}
M_{m;b_1,\ldots,b_j}\F_q[V_1,\ldots,V_n].
\end{align*}
\end{thm}


\subsection{A generalization of invariants of Minh-T\`{u}ng in $\sym\otimes\wedg$.}

The following lemma in the case $q=p$ is \cite[Corollary~1 to
Proposition~1]{MT}, where Minh and  T\`{u}ng established both
their Proposition~1 and Corollary~1 by appealing to Steenrod
algebra. We give below a simple direct proof which is valid for a
general $q$.

\begin{lem}(cf. \cite[Corollary~1]{MT})\label{lem:MTcor1}
Suppose $1\leq j\leq m\leq n$ and $0\leq b_1<\cdots<b_{j-1}<b_j= m-1$.
Then
\begin{align*}
M_{m;b_1,\ldots,b_{j-1},m-1}
(x_1,\ldots,x_{m-1},0,y_1,\ldots,y_{m-1},y_m)
=(-1)^{m+j}M_{m-1;b_1,\ldots,b_{j-1}}\cdot y_m.
\end{align*}
\end{lem}
\begin{proof}

By applying the Laplacian expansion along the last column (cf.
\cite[1.7]{Mui}) to the determinant in the definition of
$M_{m;b_1,\ldots,b_{j-1},m-1}
(x_1,\ldots,x_{m-1},0,y_1,\ldots,y_{m-1},y_m)$, we obtain that
{\allowdisplaybreaks
\begin{align*} &M_{m;b_1,\ldots,b_{j-1},m-1}
(x_1,\ldots,x_{m-1},0,y_1,\ldots,y_{m-1},y_m)\\
&=\frac{1}{j!}\sum^j_{k=1}(-1)^{m-k}(-1)^{j-k}
\left|
\begin{array}{cccc}
y_1&y_2&\cdots&y_{m-1}\\
\vdots&\vdots&\vdots&\vdots\\
y_1&y_2&\cdots&y_{m-1}\\
x_1&x_2&\cdots&x_{m-1}\\
\vdots&\vdots&\vdots&\vdots\\
\widehat{x_1^{q^{b_1}}}&\widehat{x_2^{q^{b_1}}}&\cdots&\widehat{x_{m-1}^{q^{b_1}}}\\
\vdots&\vdots&\vdots&\vdots\\
\widehat{x_1^{q^{b_{j-1}}}}&\widehat{x_2^{q^{b_{j-1}}}}&\cdots&\widehat{x_{m-1}^{q^{b_{j-1}}}}\\
\vdots&\vdots&\vdots&\vdots\\
x_1^{q^{m-2}}&x_2^{q^{m-2}}&\cdots&x_{m-1}^{q^{m-2}}
\end{array}
\right|\cdot y_m\\
&=\frac{1}{j!}(-1)^{m+j}j \big((j-1)!M_{m-1;b_1,\ldots,b_{j-1}}\big)  y_m\\
&=(-1)^{m+j}M_{m-1;b_1,\ldots,b_{j-1}}\cdot y_m.
\end{align*}
 }
The lemma is proved.
\end{proof}

Let $I=(n_1,\ldots,n_{\ell})$ be a composition of $n$. Recalling
the notations in Subsection~\ref{subsec:Dickinv} and
using~(\ref{eqn:Vm}), one can deduce that
\begin{align}
v_{i,j} =
x_{m_{i-1}+j}^{q^{m_{i-1}}}
+\sum^{m_{i-1}-1}_{k=0}(-1)^{m_{i-1}-k}Q_{m_{i-1},k}x_{m_{i-1}+j}^{q^k}.
\label{eqn:vij}
\end{align}
for $1\leq i\leq\ell$ and $1\leq j\leq n_i$. We can write $g\in
P_I$ uniquely as
\begin{equation}
g= \left( {\begin{array}{cccc}
 g_1 & *&\cdots &* \\
 0 &  g_2&\cdots& * \\
 \vdots&\vdots&\vdots&\vdots\\
 0&0&\cdots& g_{\ell}
 \end{array} }
\right)\label{eqn:elem.g}
 =\text{diag} \ (g_1, \ldots, g_\ell)
\cdot U_g
\end{equation}
where
$$
U_g = \left( {\begin{array}{cccc}
 I_{n_1} & *&\cdots &* \\
 0 &  I_{n_2}&\cdots& * \\
 \vdots&\vdots&\vdots&\vdots\\
 0&0&\cdots& I_{n_\ell}
 \end{array} }
\right).
$$
Note that $ U_g \cdot v_{i,j}=v_{i,j}$.
Since $Q_{m_{i-1},k}$ are $GL_{m_{i-1}}(q)$-invariant, using
(\ref{eqn:vij}) we have
\begin{align}
g\cdot v_{i,j}= \text{diag} \ (g_1, \ldots, g_\ell) \cdot v_{i,j}=
\sum^{n_i}_{k=1}(g_i)_{kj}v_{i,k}, \label{eqn:gvij}
\end{align}
where $(g_i)_{kj}$ denotes the $(k,j)$-entry of the matrix $g_i$.
Define $\theta_i, 1\leq i\leq\ell$, by letting
\begin{align*}
\theta_i=L_{n_i}(v_{i,1},v_{i,2},\ldots,v_{i,n_i}).
\end{align*}
Then by~(\ref{eqn:gvij}) we have
\begin{align}
g\cdot \theta_i&=\det(g_i) \ \theta_i, \label{eqn:theta}
\end{align}
for $g\in P_I$ of the form~(\ref{eqn:elem.g}) and $1\leq
i\leq\ell$.
Moreover, by~(\ref{eqn:relQL}) we have
\begin{align}
q_{i,0}=\theta_i^{q-1}.
\label{eqn:relql}
\end{align}

Let $K_I$ be the subgroup of $P_I$ consisting of matrices of the
form
\begin{align*}
\left(
{\begin{array}{cccc}
 A_1 & *&\cdots &* \\
 0 &  A_2&\cdots& * \\
 \vdots&\vdots&\vdots&\vdots\\
 0&0&\cdots& A_{\ell}
 \end{array} }
\right)
\end{align*}
with $ A_{i}\in SL_{n_i}(q), 1\leq i\leq\ell.$ Based on
Lemmas~\ref{lem:Mui4.5}-\ref{lem:Mui5.2},
Theorem~\ref{thm:Muithms}, Lemma~\ref{lem:MTcor1},
\eqref{eqn:theta} and \eqref{eqn:relql}, we can
obtain Lemma~\ref{lem:MTlem2} and Theorem~\ref{th:KI-inv} below
(which is now valid for a general $q$) by the same arguments as
in~\cite[Section~3]{MT}. We omit the details.
\begin{lem}(cf. \cite[Lemma~2]{MT})\label{lem:MTlem2}
Fix $1\leq j\leq n$ and let
\begin{align*}
f=\sum_{1\leq i\leq\ell,m_i\geq j}~\sum_{0\leq b_1<\cdots<b_j\leq
m_i-1, b_j\geq m_{i-1}}
M_{m_i;b_1,\ldots,b_j}f_{m_i;b_1,\ldots,b_j}(x_1,\ldots,x_n).
\end{align*}
Then $f=0$ if and only if all $f_{m_i;b_1,\ldots,b_j}=0$.
\end{lem}

\begin{thm}(cf. \cite[Theorem~2]{MT})  \label{th:KI-inv}
Let $I=(n_1,\ldots,n_{\ell})$ be a composition of $n$. Then,

(1) $\sym^{K_I}$ is a polynomial algebra in $n$ generators as
follows:
\begin{align*}
\sym^{K_I}=\F_q[\theta_1,q_{1,1},\ldots,q_{1,n_1-1},
\theta_2,q_{2,1},\ldots,q_{2,n_2-1},\ldots,
\theta_{\ell},q_{\ell,1},\ldots,q_{\ell,n_{\ell}-1}].
\end{align*}
(2) As an $\sym^{K_I}$-module, $(\sym\otimes\wedg)^{K_I}$ is free
and it is decomposed as
\begin{align*}
 ( & \sym  \otimes\wedg)^{K_I}\\
  & =\sym^{K_I} \oplus\sum^n_{j=1}\sum_{1\leq i\leq\ell, m_i\geq
j}~ \sum_{0\leq b_1<\cdots<b_j\leq m_i-1, b_j\geq m_{i-1}}
M_{m_i;b_1,\ldots,b_j}\sym^{K_I}.
\end{align*}
\end{thm}


\section{The  parabolic subgroup invariants in $\sym\otimes\wedg\otimes\Det^k$}
\label{sec:para.inv.}


\subsection{The $P_I$-invariants in $\sym\otimes\wedg\otimes\Det^k$.}

We shall identify the $GL_n(q)$-module
$\sym\otimes\wedg\otimes\Det^k, 1\leq k\leq q-2$, as
$\sym\otimes\wedg$ with a twisted action by the $k$th power of the
determinant, that is,
$$
g\cdot_kf=\det(g)^k \ g\cdot f, \qquad \text{ for } g\in
GL_n(q), f\in\sym\otimes\wedg.
$$

\begin{thm}\label{thm:invparabolic}
Let $I=(n_1,\ldots,n_{\ell})$ be a composition of $n$. Then,

(1)  $(\sym\otimes\wedg)^{P_I}$ is a free $\sym^{P_I}$-module of
rank $2^n$, with a basis consisting of 1 and
$M_{m_i;b_1,\ldots,b_j}\theta_1^{q-2}\cdots \theta_i^{q-2}$ for
$1\leq i\leq {\ell}, 1\leq j\leq m_i$ and $0\leq
b_1<b_2<\cdots<b_j\leq m_i-1, b_j\geq m_{i-1}$.

(2) $(\sym\otimes\wedg\otimes {\rm Det}^k)^{P_I}$ is a free
$\sym^{P_I}$-module of rank $2^n$, with a basis  consisting of
$(\theta_1\cdots\theta_{\ell})^{q-1-k}$ and
$M_{m_i;b_1,\ldots,b_j}\theta_1^{q-2-k}\cdots
\theta_i^{q-2-k}\theta_{i+1}^{q-1-k}\cdots \theta_{\ell}^{q-1-k}$
for $1\leq i\leq \ell, 1\leq j\leq m_i$ and $ 0\leq
b_1<b_2<\cdots<b_j\leq m_i-1, b_j\geq m_{i-1}$, where $1\leq k\leq
q-2$.
\end{thm}
\begin{proof}
We shall examine the $\sym^{P_I}$-module
$(\sym\otimes\wedge^j(V)\otimes{\rm Det}^k)^{P_I}$ for each fixed
$j$, starting with $j=0$.

Suppose that $1\leq k\leq q-2$. Since
$(\sym\otimes\Det^k)^{P_I}\subseteq(\sym\otimes\Det^k)^{K_I}
=\sym^{K_I}$,  it follows by Theorem~\ref{thm:KM-He},
Theorem~\ref{th:KI-inv}~(1) and (\ref{eqn:relql}) that each $f$ in
$(\sym\otimes\Det^k)^{P_I}$ can be written as
\begin{align*}
f=\sum_{0\leq \alpha_1,\ldots,\alpha_{\ell}\leq q-2}
f_{\alpha_1,\ldots,\alpha_{\ell}}
\theta_1^{\alpha_1}\cdots\theta_{\ell}^{\alpha_{\ell}}
\end{align*}
where $f_{\alpha_1,\ldots,\alpha_{\ell}}\in\sym^{P_I}$. Applying
$g\in P_I$ as in (\ref{eqn:elem.g}) to $f$ and using
(\ref{eqn:theta}), we see that
\begin{align*}
\sum_{0\leq \alpha_1,\ldots,\alpha_{\ell}\leq q-2}
 & f_{\alpha_1,\ldots,\alpha_{\ell}} \theta_1^{\alpha_1}
\cdots\theta_{\ell}^{\alpha_{\ell}}
 \\
=&\sum_{0\leq \alpha_1,\ldots,\alpha_{\ell}\leq q-2}\det(g)^k
f_{\alpha_1,\ldots,\alpha_{\ell}}(\det(g_1))^{\alpha_1}
\cdots(\det(g_{\ell}))^{\alpha_{\ell}}
\theta_1^{\alpha_1}\cdots\theta_{\ell}^{\alpha_{\ell}}
 \\
=&\sum_{0\leq \alpha_1,\ldots,\alpha_{\ell}\leq q-2}
(\det(g_1))^{k+\alpha_1}\cdots(\det(g_{\ell}))^{k+\alpha_{\ell}}
f_{\alpha_1,\ldots,\alpha_{\ell}}
\theta_1^{\alpha_1}\cdots\theta_{\ell}^{\alpha_{\ell}}
\end{align*}
and hence
\begin{align*}
(\det(g_1))^{k+\alpha_1}\cdots(\det(g_{\ell}))^{k+\alpha_{\ell}}=1,
\quad \forall g \in P_I.
\end{align*}
This implies that $k+\alpha_i, 1\leq i\leq\ell$ is divisible by
$q-1$, and thus, $k+\alpha_i=q-1$ for $1\leq i\leq\ell$ by the
constraints on $k,\alpha_i$. Therefore each
$f\in(\sym\otimes\Det^k)^{P_I}$ is of the form
\begin{align*}
f=\bar{f}\theta_1^{q-1-k}\cdots\theta_{\ell}^{q-1-k}, \quad \text{
for }  \bar{f}\in\sym^{P_I}.
\end{align*}
Hence, $(\sym\otimes\Det^k)^{P_I}$ is a free $\sym^{P_I}$-module
generated by $\theta_1^{q-1-k}\cdots\theta_{\ell}^{q-1-k}$ .

It remains to show that, for given $1\leq j\leq n$ and $0\leq
k\leq q-2$, $(\sym\otimes\wedge^j(V)\otimes{\rm Det}^k)^{P_I}$ is
a free $\sym^{P_I}$-module with basis consisting of elements
$M_{m_i;b_1,\ldots,b_j}\theta_1^{q-2}\cdots \theta_i^{q-2}$ if
$k=0$, and consisting of elements
$M_{m_i;b_1,\ldots,b_j}\theta_1^{q-2-k}\cdots
\theta_i^{q-2-k}\theta_{i+1}^{q-1-k}\cdots \theta_{\ell}^{q-1-k}$
if $1\leq k\leq q-2$, where $1\leq i\leq \ell, m_i \geq j$,
$ 0\leq b_1<b_2<\cdots<b_j\leq m_i-1,$ and $b_j\geq m_{i-1}$.

Since ${\rm Det}$ is $K_I$-invariant, by
Theorem~\ref{th:KI-inv}~(2) we have
\begin{align*}
(\sym\otimes\wedge^j(V)\otimes{\rm Det}^k)^{K_I}
&=(\sym\otimes\wedge^j(V))^{K_I}\\
&=\sum_{1\leq i\leq\ell, m_i\geq j} \sum_{0\leq b_1<\cdots<b_j\leq
m_i-1, b_j\geq m_{i-1}} M_{m_i;b_1,\ldots,b_j}\sym^{K_I}.
\end{align*}
So we can decompose $f\in
(\sym\otimes\wedge^j(V)\otimes{\Det}^k)^{P_I}$ as
\begin{align}
f=\sum_{1\leq i\leq\ell, m_i\geq j} \sum_{0\leq b_1<\cdots<b_j\leq
m_i-1, b_j\geq m_{i-1}}
M_{m_i;b_1,\ldots,b_j}f_{m_i;b_1,\ldots,b_j},\label{eqn:decompf}
\end{align}
where $f_{m_i;b_1,\ldots,b_j}\in \sym^{K_I}$. Observe that, for
$g\in P_I$ of the form~(\ref{eqn:elem.g}),
\begin{align*}
g\cdot M_{m_i;b_1,\ldots,b_j}
=\det(g_1)\cdots\det(g_i)M_{m_i;b_1,\ldots,b_j},
\end{align*}
and hence, we rewrite the equation $g\cdot_k f=f$ as
\begin{align*}
&\sum_{1\leq i\leq\ell, m_i\geq j}\sum_{0\leq b_1<\cdots<b_j\leq m_i-1, b_j\geq m_{i-1}}
M_{m_i;b_1,\ldots,b_j}f_{m_i;b_1,\ldots,b_j}\\
=&\sum_{1\leq i\leq\ell, m_i\geq j}\sum_{0\leq b_1<\cdots<b_j\leq m_i-1, b_j\geq m_{i-1}}
\det(g)^k\det(g_1)\cdots\det(g_i)
M_{m_i;b_1,\ldots,b_j}(g\cdot f_{m_i;b_1,\ldots,b_j}).
\end{align*}
This together with Lemma~\ref{lem:MTlem2} shows that each
$f_{m_i;b_1,\ldots,b_j}$ appearing in (\ref{eqn:decompf})
satisfies
\begin{align*}
f_{m_i;b_1,\ldots,b_j} &=\det(g)^k\det(g_1)
\cdots\det(g_i)(g\cdot f_{m_i;b_1,\ldots,b_j})\notag\\
&=\det(g_1)^{k+1}\cdots\det(g_i)^{k+1}\det(g_{i+1})^k
\cdots\det(g_{\ell})^k(g\cdot f_{m_i;b_1,\ldots,b_j})
\end{align*}
Recall that $f_{m_i;b_1,\ldots,b_j}\in\sym^{K_I}$. It follows by
an argument similar to the one used above in the case $j=0$ that
$f_{m_i;b_1,\ldots,b_j}$ is of the form
\begin{eqnarray} \label{eqn:divisiblef}
\left\{
 \begin{array}{ll}
 \theta_1^{q-2}\cdots\theta_i^{q-2}
\bar{f}_{m_i;b_1,\ldots,b_j}, & \text{ if
} k=0, \\
 \theta_1^{q-2-k}\cdots\theta_i^{q-2-k}
\theta_{i+1}^{q-1-k}\cdots\theta_{\ell}^{q-1-k}\bar{f}_{m_i;b_1,\ldots,b_j},
& \text{ if } 1\leq k\leq q-2
\end{array}
 \right.
\end{eqnarray}
for some $\bar{f}_{m_i;b_1,\ldots,b_j}\in\sym^{P_I}$.

We claim that both parts (1) and (2) of
Theorem~\ref{thm:invparabolic} follows now from
(\ref{eqn:decompf}), (\ref{eqn:divisiblef}), and
Lemma~\ref{lem:MTlem2}. We will only argue for (1) when $k=0$, and
skip similar arguments for (2). Indeed, for $k=0$, by
(\ref{eqn:decompf}) and (\ref{eqn:divisiblef}) each
$f\in(\sym\otimes\wedge^j(V))^{P_I}$ is of the form
\begin{align*}
\sum_{1\leq i\leq\ell, m_i\geq j} \sum_{0\leq b_1<\cdots<b_j\leq
m_i-1, b_j\geq m_{i-1}} M_{m_i;b_1,\ldots,b_j}\theta_1^{q-2}\cdots
\theta_i^{q-2}\bar{f}_{m_i;b_1,\ldots,b_j}
\end{align*}
for $\bar{f}_{m_i;b_1,\ldots,b_j}\in\sym^{P_I}$. On the other
hand, Lemma~\ref{lem:MTlem2} ensures the linear independence of
the elements as listed in Theorem~\ref{thm:invparabolic}~(1), and
any linear combination of these elements over $\sym^{P_I}$ is
clearly contained in $(\sym\otimes\wedge^j(V))^{P_I}$.
\end{proof}

\begin{rem}
Theorem~\ref{thm:invparabolic}~(1) in the special case when $q=p$
is \cite[Theorem 3]{MT}, and our arguments in the general cases
follow  closely the proof therein.
\end{rem}


\subsection{The Hilbert series of $\big(\sym\otimes\wedg\otimes {\Det}^k \big)^{P_I}$.}

The Hilbert series for the bi-graded space
$\big(\sym\otimes\wedg\otimes {\Det}^k \big)^{P_I}$ is defined to
be the generating function in $t,s$:
$$
H \big((\sym\otimes\wedg)^{P_I};t,s \big) := \sum_{i,j} t^i s^j
\dim \big(S^i(V) \otimes\wedge^j(V) \otimes {\Det}^k \big)^{P_I}.
$$
Recall the notation $m_i$ from \eqref{eq:mi}. The following
theorem and Theorem~B are complementary to each other, and they
will be proved together.

\begin{thm}\label{thm:Hser.para.}
Let $I=(n_1,\ldots,n_{\ell})$ be a composition of $n$. Then, the Hilbert
series $H \big((\sym\otimes\wedg)^{P_I};t,s \big)$ is equal to
\begin{align*}
\ds\frac{1-t^{q^{m_1}-1}+\sum^{\ell-1}_{i=1}(t^{q^{m_i}-1}-t^{q^{m_{i+1}}-1})
\prod^{m_i}_{j=1}(1+st^{-q^{j-1}})+t^{q^n-1}\prod^n_{j=1}(1+st^{-q^{j-1}})}
{\prod^{\ell}_{i=1}\prod^{n_i}_{j=1}(1-t^{q^{m_i}-q^{m_i-j}})}.
\end{align*}
\end{thm}
\begin{proof}[Proof of Theorem~\ref{thm:Hser.para.} and Theorem~B]
Denote $\deg x =(\alpha, \beta)$ for $x \in S^\alpha(V) \otimes
\wedge^\beta(V)$. Then, for $1\leq i\leq\ell$, we have
\begin{align*}
\deg M_{m_i;b_1,\ldots,b_j}
&=\Big(\frac{q^{m_i}-1}{q-1}-q^{b_1}-\ldots-q^{b_j},j\Big),
\\
\deg \theta_i& =({q^{m_{i-1}}+q^{m_{i-1}+1}+\ldots+q^{m_i-1}},0),
\end{align*}
and hence
\begin{equation}   \label{eq:degree}
\deg \theta_1\cdots\theta_i =\Big({\frac{q^{m_i}-1}{q-1}},0\Big),
 \quad
\deg \theta_{i+1}\cdots\theta_{\ell}
=\Big({\frac{q^n-1}{q-1}-\frac{q^{m_i}-1}{q-1}}, 0\Big ).
\end{equation}
We then obtain by a simple computation that, for $1\leq
i\leq\ell$,
\begin{equation}
\deg \ (M_{m_i;b_1,\ldots,b_j}(\theta_1\cdots\theta_i)^{q-2})
 =\Big({q^{m_i}-1-q^{b_1}-\ldots-q^{b_j}},j\Big),
 \label{eqn:degM0}
\end{equation}
and
\begin{align}
 \deg \ (M_{m_i;b_1, \ldots,b_j} & (\theta_1\cdots\theta_i)^{q-2-k}
(\theta_{i+1}\cdots\theta_{\ell})^{q-1-k})
 \label{eqn:degMk} \\
 =& \Big(
{\frac{q^n-1}{q-1}(q-1-k)-q^{b_1}-\ldots-q^{b_j}},j \Big). \notag
\end{align}
%

Let us now prove Theorem~\ref{thm:Hser.para.}. By
Theorem~\ref{thm:KM-He}, Theorem~\ref{thm:invparabolic}~(1) and \eqref{eqn:degM0},  we have
{\allowdisplaybreaks
\begin{align*}
&H\big((\sym\otimes\wedg)^{P_I};t,s \big)\notag\\
=&\frac{1}{\prod^{\ell}_{i=1}\prod^{n_i}_{j=1}(1-t^{q^{m_i}-q^{m_i-j}})}
~\cdot\notag\\
&\qquad \ds\Big(1+ \sum^{\ell}_{i=1}\sum_{1\leq j\leq
m_i}\sum_{0\leq b_1<\cdots<b_j\leq m_i-1, b_j\geq m_{i-1}} s^j
t^{\big(q^{m_i}-1-q^{b_1}-\ldots-q^{b_j}\big)} \Big) \\
&=\frac{1}{\prod^{\ell}_{i=1}\prod^{n_i}_{j=1}(1-t^{q^{m_i}-q^{m_i-j}})}~\cdot\\
&\qquad \ds\left(1+\sum^{\ell}_{i=1}\sum_{1\leq j\leq m_i} \sum_{0\leq
b_1<\cdots<b_j\leq m_i-1}
s^j t^{\big(q^{m_i}-1-q^{b_1}-\ldots-q^{b_j}\big)} \right. \\
&\left. \qquad\quad - \sum^{\ell}_{i=1}\sum_{1\leq j\leq m_i}
\sum_{0\leq b_1<\cdots<b_j\leq m_{i-1}-1} s^j
t^{\big(q^{m_i}-1-q^{b_1}-\ldots-q^{b_j}\big)}
\right)\\
&=\frac{1}{\prod^{\ell}_{i=1}\prod^{n_i}_{j=1}(1-t^{q^{m_i}-q^{m_i-j}})}~\cdot\\
&\qquad
 \left( 1+\sum^{\ell}_{i=1}t^{q^{m_i}-1}
\Big(\prod^{m_i}_{j=1} (1+st^{-q^{j-1}})-1\Big)
-\sum^{\ell}_{i=1}t^{q^{m_i}-1} \Big(\prod^{m_{i-1}}_{j=1}
(1+st^{-q^{j-1}})-1\Big)
\right) \\
&=\frac{1}{\prod^{\ell}_{i=1}\prod^{n_i}_{j=1}(1-t^{q^{m_i}-q^{m_i-j}})}~\cdot\\
&\Big(1+t^{q^{m_1}-1}\big(\prod^{m_1}_{j=1}(1+st^{-q^{j-1}})-1\big)+
\sum^{\ell}_{i=2}t^{q^{m_i}-1}\big(\prod^{m_i}_{j=1}(1+st^{-q^{j-1}})
-\prod^{m_{i-1}}_{j=1}(1+st^{-q^{j-1}})\big)\Big)\\
=&\ds\frac{1-t^{q^{m_1}-1}+\sum^{\ell-1}_{i=1}(t^{q^{m_i}-1}-t^{q^{m_{i+1}}-1})
\prod^{m_i}_{j=1}(1+st^{-q^{j-1}})+t^{q^n-1}\prod^n_{j=1}(1+st^{-q^{j-1}})}
{\prod^{\ell}_{i=1}\prod^{n_i}_{j=1}(1-t^{q^{m_i}-q^{m_i-j}})}.
\end{align*}
 }

Now we turn to the proof of Theorem~B. By Theorem~\ref{thm:KM-He},
Theorem~\ref{thm:invparabolic}~(2), \eqref{eq:degree} and
(\ref{eqn:degMk}), we compute the Hilbert series of
$(\sym\otimes\wedg\otimes {\Det}^k)^{P_I}$ as follows:
\begin{align*}
H \big( &(\sym\otimes\wedg\otimes {\Det}^k)^{P_I};t,s \big)\notag\\
=&H\big(\sym^{P_I};t \big)\cdot t^{(q-1-k)\frac{q^n-1}{q-1}} \notag\\
&\ds+H \big(\sym^{P_I};t \big)\cdot\sum^n_{j=1}
\sum_{1\leq i\leq\ell, m_i\geq j}~\sum_{0\leq b_1<\cdots<b_j\leq m_i-1, b_j\geq m_{i-1}}
s^j t^{\big(\frac{q^n-1}{q-1}(q-1-k)-q^{b_1}-\ldots-q^{b_j}\big)} \notag\\
=&\frac{t^{(q-1-k)\frac{q^n-1}{q-1}}}{\prod^{\ell}_{i=1}\prod^{n_i}_{j=1}(1-t^{q^{m_i}-q^{m_i-j}})}
\Big(\ds 1+\sum^n_{j=1}
\sum_{1\leq i\leq\ell, m_i\geq j}~\sum_{0\leq b_1<\cdots<b_j\leq m_i-1, b_j\geq m_{i-1}}
s^jt^{-q^{b_1}-\ldots-q^{b_j}}\Big)  
\end{align*}
which can be further simplified by noting that, for fixed $1\leq j\leq n$,
\begin{align*}
 \bigcup^{\ell}_{i=1}\{(b_1,\ldots,b_j) & |~0\leq b_1<\cdots<b_j\leq m_i-1, b_j\geq m_{i-1}\}\\
 &= \{(b_1,\ldots,b_j)~|~0\leq b_1<\cdots<b_j\leq n-1\}.
\end{align*}
Simplifying the above expression, we obtain that
%
\begin{align*}
 H\big(( & \sym\otimes\wedg\otimes {\Det}^k)^{P_I};t,s \big)\\
=&\frac{t^{(q-1-k)\frac{q^n-1}{q-1}}}{\prod^{\ell}_{i=1}
\prod^{n_i}_{j=1}(1-t^{q^{m_i}-q^{m_i-j}})}
\Big(1+\sum^n_{j=1}\sum_{0\leq b_1<\cdots<b_j\leq n-1}
s^jt^{-q^{b_1}-\ldots-q^{b_j}}\Big)\\
=&\frac{t^{(q-1-k)\frac{q^n-1}{q-1}}}{\prod^{\ell}_{i=1}
\prod^{n_i}_{j=1}(1-t^{q^{m_i}-q^{m_i-j}})}
\prod^{n-1}_{j=0}(1+st^{-q^j})\\
=&t^{(q-2-k)\frac{q^n-1}{q-1}} \cdot \frac{\prod^{n-1}_{j=0}(s+t^{q^j})}
{\prod^{\ell}_{i=1}\prod^{n_i}_{j=1}(1-t^{q^{m_i}-q^{m_i-j}})}.
\end{align*}
Hence, Theorem~B is proved.
\end{proof}


\subsection{The $P_I$-invariants in $\sym\otimes\wedg^*$}

There is an isomorphism of $GL_n(q)$-modules $\wedge^{n-j}(V)\cong
\wedge^{j}(V)^*\otimes\wedge^n(V)$ and $\wedge^n(V)\cong \Det$,
where $W^*$ denotes the dual module of $W$. Hence, we have an
isomorphism of $GL_n(q)$-modules
\begin{equation} \label{eq:iso}
\wedge^j(V)^* \cong \wedge^{n-j}(V)\otimes\Det^{q-2}, \qquad 0\leq j\leq n.
\end{equation}

Observe that $\sym\otimes \wedg^*\cong {\rm Hom} (\wedg,
\sym)\cong \oplus_{0\leq m\leq n}{\rm Hom} (\wedge^{m}(V), \sym)$
as $GL_n(q)$-modules. Moreover,
$$
A^*(V;\sym):=\oplus_{0\leq m\leq n}{\rm Hom} (\wedge^m(V), \sym)
$$
affords an $\F_q$-algebra structure as follows. The product
$\alpha\wedge\beta\in{\rm Hom} (\wedge^{m+e}(V), \sym)$, for
$\alpha\in{\rm Hom} (\wedge^m(V), \sym),\beta\in{\rm Hom}
(\wedge^e(V), \sym)$,  is defined by
\begin{align*}
(\alpha\wedge\beta)(v_1\wedge v_2\wedge\cdots\wedge v_{m+e})
=\sum {\rm sgn}(\sigma)\alpha(v_{\sigma(1)}\wedge\cdots\wedge v_{\sigma(m)})
\beta(v_{\sigma(m+1)}\wedge\cdots\wedge v_{\sigma(m+e)}),
\end{align*}
where the summation is over all permutations $\sigma\in S_{m+e}$ satisfying
$\sigma(1)<\cdots<\sigma(m)$ and $\sigma(m+1)<\cdots<\sigma(m+e)$.

Following Crabb \cite{Cr}, we define $\omega_i(V)\in {\rm Hom} (V,
\sym)$ to be the linear map which sends each $v \in V$ to
$v^{q^i}$, for $i\geq 0$. Then, each $\omega_i(V)$ is
$GL_n(q)$-invariant, and one shows \cite{Cr} that
$\omega_0(V),\ldots,\omega_{n-1}(V)$ generate an exterior
$\F_q$-subalgebra of $A^*(V;\sym)$, which is denoted by $
\Omega(V)$. The following proposition recovers
\cite[Proposition~1.1]{Cr}, when $I =(n)$ and so $P_I =GL_n(q)$.
\begin{prop}
We have the following isomorphism of (bi-graded) $\F_q$-algebras:
\begin{align*}
A^*(V;\sym)^{P_I}=\Omega(V)\otimes_{\F_q}\sym^{P_I}.
\end{align*}
\end{prop}
\begin{proof}
Clearly, $\Omega(V)\otimes_{\F_q}\sym^{P_I}\subseteq
A^*(V;\sym)^{P_I}$. By Theorem~\ref{thm:KM-He} and the definition
of $\omega_i(V)$, the Hilbert series of the bi-graded algebra
$\Omega(V)\otimes_{\F_q}\sym^{P_I}$ is
\begin{align*}
H \big(\Omega(V)\otimes_{\F_q}\sym^{P_I};t,s \big)
=\frac{\prod^{n-1}_{i=0}(1+st^{q^i})}{\prod^{s}_{i=1}
\prod^{n_i}_{j=1}(1-t^{q^{m_i}-q^{m_i-j}})}
\end{align*}
where we recall the notation $m_i$ from \eqref{eq:mi}. It readily
follows by Theorem~B and \eqref{eq:iso} that the Hilbert series of
$A^*(V;\sym)^{P_I}$ is given by the same answer. The proposition
follows.
\end{proof}


\section{The Steinberg module multiplicity in $\sym\otimes\wedg\otimes {\Det}^k$}
\label{sec:Stmult}

In the section we shall prove Theorem~C on the graded multiplicity
of the Steinberg module $\St$ in the bi-graded
$GL_n(q)$-module $\sym\otimes\wedg\otimes\Det^k$ for $0\leq k\leq
q-2$.

For a (bi-)graded $GL_n(q)$-module $N^{\bullet}
=\oplus_{i,j}N^{i,j}$, we denote by $H_{\St}(N^{\bullet};t,s)$ the
graded multiplicity for the Steinberg module $\St$ in $N$. We have
\begin{align*}
H_{\St}(N^{\bullet};t,s)=\sum_{i,j}  t^i s^j \dim
\text{Hom}_{GL_n(q)} (\St, N^{i,j})
\end{align*}
since the Steinberg module $\St$
is absolutely irreducible and projective. For more on
the Steinberg module, we refer to the excellent book of Humphreys \cite{Hu}.

It was observed in \cite[Corollary 1.3]{KM} that, for any
$GL_n(q)$-module $N$,
\begin{align}
{\rm dim~Hom}_{GL_n(q)}(\St,N)=
\sum_{I=(n_1,\ldots,n_{\ell})}(-1)^{n-\ell}{\rm dim}~N^{P_I},
\label{eqn:KM}
\end{align}
where the summation is over all compositions $I$ of $n$ (recall
our convention for composition requires all parts $n_i>0$). The
characteristic zero counterpart of \eqref{eqn:KM} was originally
due to Curtis \cite{Cu} and its geometric and homological
interpretation was given by Solomon \cite{So}. By
Theorem~\ref{thm:KM-He} and \eqref{eqn:KM}, Kuhn and
Mitchell~\cite{KM} showed that
\begin{align}
H_{\St}(\sym;t) &=\sum_{I=(n_1,\ldots,n_{\ell})}(-1)^{n-\ell}
\frac{1}{\prod^{\ell}_{i=1}\prod^{n_i}_{j=1}(1-t^{q^{m_i}-q^{m_i-j}})}\notag\\
&=\frac{t^{-n+\frac{q^n-1}{q-1}}}{\prod^n_{i=1}(1-t^{q^i-1})},
\label{eqn:KMthm}
\end{align}
where \eqref{eqn:KMthm} is a nontrivial combinatorial identity.

\begin{rem}  \label{rem:MP}
Parts~(1) and (2) of Theorem~C in the case when $q$ is a prime
were established by Mitchell-Priddy \cite{MP} and Mitchell
\cite{Mi}, respectively, via a topological approach which involves
Steenrod algebra. A factor $t^{-n}$ in the numerator is missing in
the original formula~\cite[Theorem~2.6]{Mi}.
\end{rem}

Now we are ready to prove Theorem~C.
\begin{proof}[Proof of Theorem~C]
Let us first prove the easier Part (2). Recall the notation $m_i$
from \eqref{eq:mi}. By Theorem~B, the identities
(\ref{eqn:KM})~and~(\ref{eqn:KMthm}), we have
\begin{align*}
H_{\St} & \big(\sym\otimes\wedg\otimes {\rm Det}^k;t,s \big)\\
=&\ds\sum_{I=(n_1,\ldots,n_{\ell})}(-1)^{n-\ell}
H((\sym\otimes\wedg\otimes {\rm Det}^k)^{P_I};s,t)\\
=&\ds t^{(q-2-k)\frac{q^n-1}{q-1}}
\prod^{n-1}_{i=0}(s+t^{q^i})~\sum_{I=(n_1,\ldots,n_{\ell})}(-1)^{n-\ell}
\frac{1}{\prod^{\ell}_{i=1}\prod^{n_i}_{j=1}(1-t^{q^{m_i}-q^{m_i-j}})}\\
=&\ds t^{(q-2-k)\frac{q^n-1}{q-1}}\prod^{n-1}_{i=0}(s+t^{q^i})
\frac{t^{-n+\frac{q^n-1}{q-1}}}{\prod^n_{i=1}(1-t^{q^i-1})}\\
=&\ds t^{-n+(q-1-k)\frac{q^n-1}{q-1}} \cdot
\frac{\prod^{n-1}_{i=0}(s+t^{q^i})}{\prod^n_{i=1}(1-t^{q^i-1})}.
\end{align*}

Now we turn to the proof of (1). It follows by
Theorem~\ref{thm:Hser.para.} and  (\ref{eqn:KM}) that
\begin{align*}
H_{\St} & \big(\sym\otimes\wedg;t,s \big)  \\
=& \ds\sum_{I=(n_1,\ldots,n_{\ell})}
(-1)^{n-\ell}~H((\sym\otimes\wedg)^{P_I};t,s) =X(t,s; q)
\end{align*}
where we have denoted
\begin{align*}
X& (t,s; q) \\
 =& \ds\sum_{I=(n_1,\ldots,n_{\ell})}(-1)^{n-\ell}~
\frac{1-t^{q^{m_1}-1}+\sum^{\ell-1}_{i=1}(t^{q^{m_i}-1}-t^{q^{m_{i+1}}-1})
\prod^{m_i}_{j=1}(1+st^{-q^{j-1}})}
{\prod^{\ell}_{i=1}\prod^{n_i}_{j=1}(1-t^{q^{m_i}-q^{m_i-j}})}
  \\
& \quad +\sum_{I=(n_1,\ldots,n_{\ell})}
(-1)^{n-\ell}~\frac{t^{q^n-1}\prod^n_{j=1}(1+st^{-q^{j-1}})}
{\prod^{\ell}_{i=1}\prod^{n_i}_{j=1}(1-t^{q^{m_i}-q^{m_i-j}})}.
\end{align*}
Hence, Part~(1) of Theorem~C is clearly equivalent to the validity
of the following combinatorial identity:
\begin{equation} \label{identity}
X (t,s; q)
 = \ds t^{-n} \cdot\frac{ (st^{q^n-1}+t^{q^{n-1}})
\prod^{n-2}_{i=0}(s+t^{q^i})}{\prod^n_{i=1}(1-t^{q^i-1})}.
\end{equation}

We do not have a direct proof of this combinatorial identity
\eqref{identity}, so we proceed in a roundabout way as follows. By
Remark~\ref{rem:MP}, Theorem~C~(1), and hence the identity
\eqref{identity}, is known to be true when $q$ is an arbitrary
prime. If we clear the denominator in \eqref{identity}, this
combinatorial identity is a polynomial equation in $t, s$, and
either side of this polynomial equation is a finite linear
combination of monomials in $t,s$ (the number of such monomials is
independent of $q$). This polynomial equation  holds for
infinitely many positive integer values $q$ (i.e., the primes),
and thus it must be a universal identity valid for every
positive integer $q$ (and in particular valid for every prime
power $q$). Hence, Part~(1) of Theorem~C is proved.
\end{proof}

\begin{rem}
It will be very interesting to find a direct proof of the
identity \eqref{identity}.
\end{rem}


\end{document}